\documentclass[11pt]{article}

\usepackage[margin=1.15in]{geometry}
\usepackage{amsmath,amssymb,amsthm,mathtools}
\usepackage{aliascnt}
\usepackage[hidelinks]{hyperref}
\usepackage[nameinlink,capitalise,noabbrev]{cleveref}
\usepackage{microtype}

\newtheorem{theorem}{Theorem}

\newaliascnt{proposition}{theorem}
\newtheorem{proposition}[proposition]{Proposition}
\aliascntresetthe{proposition}

\newaliascnt{lemma}{theorem}
\newtheorem{lemma}[lemma]{Lemma}
\aliascntresetthe{lemma}

\newaliascnt{corollary}{theorem}
\newtheorem{corollary}[corollary]{Corollary}
\aliascntresetthe{corollary}

\theoremstyle{remark}
\newaliascnt{remark}{theorem}
\newtheorem{remark}[remark]{Remark}
\aliascntresetthe{remark}

\newcommand{\Ezero}{E_0}
\newcommand{\ran}{\operatorname{ran}}
\newcommand{\symdiff}{\mathbin{\triangle}}
\newcommand{\Aut}{\operatorname{Aut}}

\title{Two questions on $E_0$-like generic equivalence}
\author{Jason Zesheng Chen}
\date{September 2026}

\begin{document}

\maketitle

\begin{abstract}
We answer Problems 8.1 and 8.3 from \cite{Tianyuan2026}.  We first note that the condition attributed there to ordinary Prikry forcing is not correct as written: the appropriate
formulation uses finite symmetric difference of the ranges of the generic
sequences, rather than eventual equality at the same coordinates.  We record
answers to the two problems for both formulations.  A length-$\omega$
Magidor forcing gives a genuinely Prikry-type example satisfying both
conditions.  Rigid real-adding forcing gives a broad source of further
examples, and a forcing of Jech and Shelah shows that the condition does not carry any large cardinal strength.  An $E_0$-invariant Jensen-type forcing of Kanovei and
Lyubetsky gives a stronger nontrivial example in which the generic reals in a
fixed extension form exactly one full $E_0$-class.  Finally, a simple recoding turns
the real-forcing previous examples into answers for the version of the problems with corrected condition.
\end{abstract}

\section{The two formulations}

Following \cite{Smythe}, consider a countable model $M$ of a
sufficiently large fragment of ZFC and a forcing notion $\mathbb P\in M$. We are interested in the equivalence relation $G \sim H\Leftrightarrow M[G]=M[H]$.
Section 8 of \cite{Tianyuan2026} takes the analysis of Prikry forcing (\cite{CS}) as its motivating example, where the generic filter is identified with an
increasing generic sequence
\[
    a=\langle a(n):n<\omega\rangle.
\]
The condition printed there is
\begin{equation*}
\tag{$\ast_{\mathrm{coord}}$}
 M[a]=M[b]
 \quad\Longleftrightarrow\quad
 \exists N<\omega\ \forall n>N\ (a(n)=b(n)).
\end{equation*}
This condition is not correct for ordinary Prikry forcing.\footnote{If
$a=\langle\alpha_0,\alpha_1,\ldots\rangle$ is Prikry generic over a
transitive $M$, then its shift
$b=\langle\alpha_1,\alpha_2,\ldots\rangle$ is again Prikry generic.  Moreover
$M[a]=M[b]$: one inclusion is immediate, and the other follows because
$\alpha_0\in M$.  Since $a$ is strictly increasing, however,
$a(n)\neq b(n)$ for every $n$.}

The corresponding statement with ranges is
\begin{equation*}
\tag{$\ast_{\mathrm{range}}$}
 M[a]=M[b]
 \quad\Longleftrightarrow\quad
 \bigl|\ran(a)\symdiff\ran(b)\bigr|<\omega.
\end{equation*}
This is the most plausible reading of Calderoni--Sinapova's phrasing, phrase
``coincide on a tail,''
\cite[Fact 4.2(1)]{CS}: the proof of the nontrivial direction assumes that
the symmetric difference is infinite and derives a contradiction.  The
converse is immediate for increasing sequences over a transitive ground
model: if their ranges differ by finitely many ordinals, each sequence is
recoverable from the other.

In \cite{Tianyuan2026}, Problem 8.1 asks for more examples of forcing that satisfy the condition; Problem 8.3 asks if the condition has large cardinal strength. For both the $(\ast_{\mathrm{coord}})$ and $(\ast_{\mathrm{range}})$ formulations of the problems, we supply further examples for 8.1 and answer 8.3 negatively.

\section{A Magidor-forcing example}

The closest example to the intended Prikry example is supplied by the
uniqueness theorem for Magidor sequences.  We use the formulation of Fuchs
\cite{Fuchs}.  Let
\[
    \mathbb M=\mathbb M(\vec U,\vec f)
\]
be a Magidor forcing of length $\alpha$, based on a measurable cardinal
$\kappa$ carrying a sequence
$\vec U=\langle U_\xi:\xi<\alpha\rangle$ of normal measures increasing in
the Mitchell order.  Its generic object is an increasing function
$c:\alpha\to\kappa$.

\begin{theorem}[Fuchs]\label{thm:fuchs}
The following is \cite[Theorem 7.5 and Corollary 7.6]{Fuchs}.  Let $c$ be $\mathbb M$-generic over a transitive ground model $M$, and let
$d\in M[c]$ be a function of the type of an $\mathbb M$-generic sequence.
Then the following are equivalent:
\begin{enumerate}
    \item $d$ is $\mathbb M$-generic over $M$;
    \item the set
    \[
        \{\xi<\alpha:c(\xi)\neq d(\xi)\}
    \]
    is finite and contains no limit ordinal;
    \item $M[c]=M[d]$.
\end{enumerate}
\end{theorem}

\begin{proposition}\label{prop:magidor}
Suppose that $M$ has a measurable cardinal $\kappa$ carrying a
Mitchell-increasing sequence $\langle U_n:n<\omega\rangle$ of normal
measures, and let $\mathbb M\in M$ be the corresponding Magidor forcing of
length $\omega$.  For any two $\mathbb M$-generic sequences $c,d$ over $M$,
\[
    M[c]=M[d]
    \quad\Longleftrightarrow\quad
    \exists N<\omega\ \forall n>N\ (c(n)=d(n)).
\]
Moreover,
\[
    M[c]=M[d]
    \quad\Longleftrightarrow\quad
    \bigl|\ran(c)\symdiff\ran(d)\bigr|<\omega.
\]
Thus length-$\omega$ Magidor forcing satisfies both
$(\ast_{\mathrm{coord}})$ and $(\ast_{\mathrm{range}})$.
\end{proposition}

\begin{proof}
Suppose first that $M[c]=M[d]$.  Then $d\in M[c]$, so
\Cref{thm:fuchs} applies.  It follows in particular that $c(n)=d(n)$ for all but finitely many
$n$.  This also
implies that the ranges of $c$ and $d$ have finite symmetric difference.

Conversely, suppose that $c$ and $d$ differ at only finitely many
coordinates.  The finite set
\[
    F=\{(n,d(n)):c(n)\neq d(n)\}
\]
belongs to $M$, and $d$ is definable from $c$ and $F$; the same argument with
$c$ and $d$ interchanged gives $M[c]=M[d]$.  Likewise, suppose that the
ranges have finite symmetric difference.  The finite sets
\[
    A=\ran(c)\setminus\ran(d)
    \quad\text{and}\quad
    B=\ran(d)\setminus\ran(c)
\]
belong to $M$, and $d$ is the increasing enumeration of
$(\ran(c)\setminus A)\cup B$.  Again $M[c]=M[d]$.
\end{proof}

This gives a positive answer to Problem 8.1, in both formulations, within
the class of genuinely Prikry-type forcings.

\section{Rigid forcing}

There is also a general source of ZFC-examples.  Recall the Vop\v{e}nka--H\'ajek theorem in the form
used by Smythe.

\begin{theorem}[Vop\v{e}nka--H\'ajek]
\label{thm:VH}
This is the form recorded in \cite[Theorem 2.16]{Smythe}.  Let $\mathbb B\in M$ be a complete Boolean algebra, and let $G,H$ be
$\mathbb B$-generic filters over $M$.  If $M[G]=M[H]$, then there is an
involutive automorphism $e\in\Aut^M(\mathbb B)$ such that
$H=e''G$.
\end{theorem}

A complete Boolean algebra is \emph{rigid} if it has no nonidentity
automorphism.

\begin{proposition}\label{prop:rigid-general}
Let $M$ be a countable transitive model, let $\mathbb B\in M$ be a complete
Boolean algebra which is rigid in $M$, and suppose that $\dot x$ is a
$\mathbb B$-name such that
\[
    M[G]=M[x_G]
\]
for every $\mathbb B$-generic filter $G$ over $M$, where
$x_G=\dot x^G\in\omega^\omega$.  Then for any two such generic filters
$G,H$,
\[
    M[x_G]=M[x_H]
    \quad\Longleftrightarrow\quad
    x_G=x_H
    \quad\Longleftrightarrow\quad
    x_G\,\Ezero\,x_H.
\]
In particular, the generic reals satisfy $(\ast_{\mathrm{coord}})$.
\end{proposition}

\begin{proof}
If $M[x_G]=M[x_H]$, then $M[G]=M[H]$.  By \Cref{thm:VH}, there is an
automorphism $e\in\Aut^M(\mathbb B)$ with $H=e''G$.  Since $\mathbb B$ is
rigid in $M$, $e$ is the identity; hence $G=H$ and $x_G=x_H$.

Equality clearly implies eventual equality.  Conversely, if
$x_G\,\Ezero\,x_H$, then the two reals differ on only finitely many
coordinates.  The finite modification taking one real to the other is coded by an
element of $M$, so $x_G$ and $x_H$ are interdefinable over $M$.  Thus
$M[x_G]=M[x_H]$, and the first part of the proof then gives $x_G=x_H$.
\end{proof}

The hypothesis of this proposition is realized in ZFC.

\begin{theorem}[Jech--Shelah \cite{JS}]
\label{thm:JS}
There is a perfect-tree forcing $\mathbb P$ which adds a generic branch
$g\in\omega^\omega$ such that
\[
    V[G]=V[g]
\]
and whose complete Boolean algebra $\mathbb B(\mathbb P)$ is rigid.
\end{theorem}

\begin{corollary}\label{cor:rigid}
For every suitable countable transitive model $M$, the Jech--Shelah forcing
as computed in $M$ satisfies $(\ast_{\mathrm{coord}})$. 
Consequently, Problem 8.1 with the printed condition has a positive answer,
whereas Problem 8.3 has a negative answer: the condition does not even imply
that $M$ has an inaccessible cardinal.
\end{corollary}

\begin{proof}
Apply \Cref{prop:rigid-general} to the Boolean completion of the
Jech--Shelah forcing.  The construction and the proof of rigidity are
carried out in ZFC and hence can be performed inside any sufficiently strong
$M$.  For the ground model, reflect the required finite fragment together
with the sentence ``there is no inaccessible cardinal'' in an appropriate
constructible model, and then take a countable elementary submodel and its
transitive collapse.  By condensation the resulting model is some
$L_\theta$ with the required properties.''
\end{proof}

\begin{remark}
For a rigid forcing, equality of generic extensions is equality of the
corresponding generic filters, and in the situation above it is equality of
the generic reals.  Thus the example is literal but in this sense rigid:
the same-extension classes of generic reals are singletons.  The next
example shows that this degeneracy is not necessary.
\end{remark}

\section{A nontrivial \texorpdfstring{$E_0$}{E0}-example}

Kanovei and Lyubetsky \cite{KL} construct in $L$ an $E_0$-invariant
Jensen-type perfect-tree forcing based on Silver forcing.  The feature we
need is their Lemma 7.4: if $G$ is generic for their forcing over $L$ and
$x_G$ is the generic real, then, inside $L[G]$, the reals generic for the
same forcing over $L$ are exactly
\[
    [x_G]_{\Ezero}.
\]
They also have $L[G]=L[x_G]$.

For the formulation of Section 8, let $M=L_\theta$ be a countable
transitive model of a fixed finite fragment of ZFC strong enough to carry
out the Kanovei--Lyubetsky construction and its proof through Lemma 7.4.
Perform their construction internally to $M$, through $\omega_1^M$, and
denote the resulting forcing by $\mathbb P^M$.

\begin{lemma}
\label{lem:KLrel}
Let $G\subseteq\mathbb P^M$ be generic over $M$, and let $x_G$ be its
generic real.  Then $M[G]=M[x_G]$, and
\[
 \{y\in M[G]\cap2^\omega:
       y\text{ is }\mathbb P^M\text{-generic over }M\}
   =[x_G]_{\Ezero}.
\]
\end{lemma}

\begin{proof}
The proofs of Lemmas 7.1 and 7.4 of \cite{KL} use the internally constructed
sequence of forcing notions and the maximal antichains belonging to the
ground model.  The same proofs therefore apply with $L$ replaced by $M$.
\end{proof}

\begin{proposition}\label{prop:KL}
If $x,y$ are $\mathbb P^M$-generic reals over $M$, then
\[
    M[x]=M[y]
    \quad\Longleftrightarrow\quad
    x\,\Ezero\,y.
\]
Consequently $\mathbb P^M$ satisfies $(\ast_{\mathrm{coord}})$.
\end{proposition}

\begin{proof}
If $M[x]=M[y]$, then $y\in M[x]$ and $y$ is
$\mathbb P^M$-generic over $M$.  By \Cref{lem:KLrel},
$y\in[x]_{\Ezero}$.

Conversely, if $x\,\Ezero\,y$, then $x$ and $y$ differ by a finite bit
flip coded in $M$.  Hence each belongs to the model generated by the other,
so $M[x]=M[y]$.
\end{proof}

\begin{corollary}\label{cor:KL-lc}
The printed condition $(\ast_{\mathrm{coord}})$ can hold nontrivially over
a model with no inaccessible cardinal: in a $\mathbb P^M$-generic
extension, the generic reals over $M$ belonging to that extension form the
entire countable $E_0$-class of the generic real.
\end{corollary}

\begin{proof}
Choose a countable transitive
\[
    M=L_\theta\models
    V=L+\text{``there is no inaccessible cardinal.''}
\]
which satisfies the finite fragment required for \Cref{lem:KLrel}.  Such a
model is obtained by the same reflection-and-collapse argument used in
\Cref{cor:rigid}.  Now apply \Cref{prop:KL}.
\end{proof}

\section{Recoding real forcings for the range condition}

The real-forcing examples above can be converted uniformly into examples
for $(\ast_{\mathrm{range}})$.  For $x\in\omega^\omega$, define
\[
    c_x(n)=\omega\cdot n+x(n).
\]
Then $c_x$ is a strictly increasing sequence of ordinals below $\omega^2$,
and $x$ and $c_x$ are uniformly definable from one another.  Moreover,
\begin{equation}\label{eq:range-code}
    x\,\Ezero\,y
    \quad\Longleftrightarrow\quad
    \bigl|\ran(c_x)\symdiff\ran(c_y)\bigr|<\omega.
\end{equation}
Indeed, $\ran(c_x)$ contains exactly one point in the ordinal interval
$[\omega\cdot n,\omega\cdot(n+1))$, namely $\omega\cdot n+x(n)$.

For a finite sequence $s\in\omega^{<\omega}$, let
\[
    \widehat s(i)=\omega\cdot i+s(i).
\]
If $T\subseteq\omega^{<\omega}$ is a tree, put
\[
    \widehat T=\{\widehat s:s\in T\}\subseteq(\omega^2)^{<\omega}.
\]
Thus any tree forcing $\mathbb P$ on $\omega^{<\omega}$ has an isomorphic
presentation
\[
    \widehat{\mathbb P}=\{\widehat T:T\in\mathbb P\}
\]
whose generic branch corresponding to $x$ is $c_x$.  The transformation is
uniform, so $\widehat{\mathbb P}\in M$ whenever $\mathbb P\in M$.

\begin{proposition}[Recoding lemma]\label{prop:recoding}
Let $\mathbb P\in M$ be a tree forcing whose generic branch
$x\in\omega^\omega$ generates the generic extension, and suppose that for
any two $\mathbb P$-generic branches $x,y$ over $M$,
\[
    M[x]=M[y]
    \quad\Longleftrightarrow\quad
    x\,\Ezero\,y.
\]
Then the isomorphic forcing $\widehat{\mathbb P}$ has a strictly increasing
generic branch $c_x$ and satisfies
\[
    M[c_x]=M[c_y]
    \quad\Longleftrightarrow\quad
    \bigl|\ran(c_x)\symdiff\ran(c_y)\bigr|<\omega
\]
for all $\widehat{\mathbb P}$-generic branches $c_x,c_y$ over $M$.
\end{proposition}

\begin{proof}
Since $x$ and $c_x$ are uniformly interdefinable,
$M[x]=M[c_x]$.  The conclusion now follows from the hypothesis and
\eqref{eq:range-code}.
\end{proof}

\begin{corollary}\label{cor:range}
Both the Jech--Shelah forcing and the Kanovei--Lyubetsky forcing have
isomorphic presentations satisfying $(\ast_{\mathrm{range}})$.  These
presentations can be constructed over models of just plain ZFC. Hence Problem 8.1 has a
positive answer and Problem 8.3 has a negative answer also for the corrected
range condition.
\end{corollary}

\begin{proof}
For the Jech--Shelah forcing, \Cref{prop:rigid-general} gives the
hypothesis of \Cref{prop:recoding}; for the Kanovei--Lyubetsky forcing, the
hypothesis is \Cref{prop:KL}.  Apply \Cref{prop:recoding}.  The ground model
may be chosen to satisfy the finite fragments needed for both constructions,
as in \Cref{cor:rigid,cor:KL-lc}.
\end{proof}

\begin{remark}
The recoded examples settle the corrected problem under the unrestricted
reading in which one asks for a forcing notion with a distinguished
increasing generic sequence.  They are isomorphic presentations of real
forcings.  By contrast, \Cref{prop:magidor} gives a direct example whose
generic sequence is genuinely Magidor generic and cofinal in the relevant
measurable cardinal.
\end{remark}


\begin{thebibliography}{9}

\bibitem{Tianyuan2026}
G.~Barmpalias, S.~Gao, J.~He, T.~Kihara, A.~Nies, D.~Schrittesser,
T.~Slaman, C.-M.~Tran, D.~Turetsky, P.~Welch, L.~Yu, and H.~Zhang,
\emph{Open Problems in Mathematical Logic},
arXiv:2608.26628, 2026.

\bibitem{CS}
F.~Calderoni and D.~Sinapova,
\emph{Forcing, genericity, and CBERS},
arXiv:2503.14811, 2025.

\bibitem{Fuchs}
G.~Fuchs,
\emph{On sequences generic in the sense of Magidor},
J. Symbolic Logic \textbf{79} (2014), no.~4, 1286--1314.

\bibitem{JS}
T.~Jech and S.~Shelah,
\emph{A complete Boolean algebra that has no proper atomless complete
subalgebra},
J. Algebra \textbf{182} (1996), no.~3, 748--755.

\bibitem{KL}
V.~Kanovei and V.~Lyubetsky,
\emph{A definable $E_0$-class containing no definable elements},
Arch. Math. Logic \textbf{54} (2015), 711--723.

\bibitem{Smythe}
I.~B. Smythe,
\emph{Equivalence of generics},
Arch. Math. Logic \textbf{61} (2022), 795--812.

\end{thebibliography}
\end{document}